\documentclass[10pt]{article}
\font\smallit=cmti10

\usepackage{amssymb,latexsym,amsmath,epsfig,amsthm} 

\makeatletter

\renewcommand\section{\@startsection {section}{1}{\z@}
{-30pt \@plus -1ex \@minus -.2ex}
{2.3ex \@plus.2ex}
{\normalfont\normalsize\bfseries}}

\renewcommand\subsection{\@startsection{subsection}{2}{\z@}
{-3.25ex\@plus -1ex \@minus -.2ex}
{1.5ex \@plus .2ex}
{\normalfont\normalsize\bfseries}}

\renewcommand{\@seccntformat}[1]{\csname the#1\endcsname. }

\makeatother

\newtheorem{theorem}{Theorem}
\newtheorem{lemma}{Lemma}

\theoremstyle{definition}
\newtheorem{remark}{Remark}
\newtheorem{definition}{Definition}

\begin{document}

\begin{center}
\uppercase{\bf A generalization of an identity due to Kimura and Ruehr}
\vskip 20pt
{\bf J.-P. Allouche} \\
{\smallit 
CNRS, Institut de Math\'ematiques de Jussieu-PRG \\
Universit\'e Pierre et Marie Curie, Case 247 \\
4 Place Jussieu \\
F-75252 Paris Cedex 05 France} \\
{\tt jean-paul.allouche@imj-prg.fr}\\
\end{center}
\vskip 30pt


\centerline{\bf Abstract}
\noindent
An identity stated by Kimura and proved by Ruehr, Kimura and others stipulates that
for any function $f$ continuous on $[-\frac{1}{2}, \frac{3}{2}]$ one has
$$
\int_{-1/2}^{3/2} f(3x^2 - 2x^3) dx = 2 \int_0^1 f(3x^2 - 2x^3) dx.
$$
We prove that this equality is not an isolated example by providing a family of polynomials,
related to the Tchebychev polynomials and of which $(3x^2 - 2x^3)$ is a particular case,
giving rise to similar identities.


\vskip 30pt

\section{Introduction}

In this text, we address an identity that we call the Kimura-Ruehr identity:
this was a question posed by Kimura and answered by Ruehr, but also by the proposer 
as well as by nine other contributors; see \cite{Kimura-Ruehr}. It reads

\medskip

{\it Let $f$ be a real function that is continuous on $[-\frac{1}{2}, \frac{3}{2}]$. Then
\begin{equation}\label{ruehr1}
\int_{-1/2}^{3/2} f(3x^2 - 2x^3) dx = 2 \int_0^1 f(3x^2 - 2x^3) dx.
\end{equation}
} 

\medskip

In his proof \cite{Kimura-Ruehr}, Ruehr notes that the identity is equivalent to the identities 
obtained for $f(x) = x^n$ for all nonnegative integers. In particular, he points out the identities
\begin{equation}\label{ruehr2}
\sum_{0 \leq j \leq n} 3^j {3n-j \choose 2n} = 
\sum_{0 \leq j \leq 2n} (-3)^j {3n-j \choose n}
\end{equation}
and
\begin{equation}\label{ruehr3}
\sum_{0 \leq j \leq n} 2^j {3n+1 \choose n-j} =
\sum_{0 \leq j \leq 2n} (-4)^j {3n+1 \choose n+1+j}.
\end{equation}
Equality~(\ref{ruehr2}) is the corrected version of the corresponding one given in 
\cite{Kimura-Ruehr}, as indicated in \cite{MTWZ} (also see \cite{Alzer-Prodinger}).

\bigskip

A way to generalize these Identities~(\ref{ruehr2}) and (\ref{ruehr3}) is to 
introduce polynomials with binomial coefficients whose values at some point
coincide with the quantities above: this was done in \cite{Alzer-Prodinger},
and, with two extra parameters, in \cite{KA}.

\medskip

Now another question that quickly comes to mind when looking at Equality~(\ref{ruehr1})
is whether this equality is ``isolated'', or whether it is an instance in a general family
of identities. Here we give a countable family of equalities that generalize Equality~(\ref{ruehr1}): 
they are somehow based on trigonometry (actually on the use of Tchebychev polynomials), in relation 
to the spirit of Ruehr's original proof.

\section{Definitions}\label{def}

Recall that the Tchebychev polynomials of the first kind are the polynomials $T_n(X)$
defined by $T_0(X) = 1$, $T_1(X) = X$, and for all $n \geq 0$, $T_{n+2}(X) = 2XT_{n+1}(X) - T_n(X)$.
They have the property that, for all $\theta \in {\mathbb R}$, the relation $T_n(\cos \theta) = \cos n \theta$ holds.

\medskip

In the rest of the paper we will use the following quantities.

\begin{definition}

\ { }

\begin{itemize}

\item For each integer $n > 1$, $a_n$ and $b_n$ are defined by
$$
a_n := \cos^2\frac{\pi}{n} - \cos^2\frac{\pi}{2n} = 
\frac{1}{2}\left(\cos \frac{2\pi}{n}  - \cos \frac{\pi}{n}\right)
\ \ \ \mbox{\rm and} \ \ \
b_n := \cos^2\frac{\pi}{2n} = \frac{1}{2}\left(\cos \frac{\pi}{n} + 1 \right).
$$

\item Furthermore, let $f$ be a function that is continuous on $[0,1]$.  For
$n > 1$ we let $A_n(f)$ and $B_n(f)$ denote the two quantities
$$
A_n(f) := \frac{1}{a_n} \int_0^{\pi/2n} f(\cos^2 nu) \sin 2u \ \mbox{\rm d}u 
\ \ \ \mbox{\rm and} \ \ \
B_n(f) := \frac{1}{b_n} \int_0^{\pi/2n} f(\cos^2 nu) \cos 2u \ \mbox{\rm d}u.
$$
\end{itemize}

\end{definition}

\begin{remark}
Note that $a_n < 0$.
\end{remark}

\begin{definition}
We define the polynomials $V_n(X)$ and $W_n(X)$ by
$$
V_n(X^2) := T_n^2(X)
\ \ \ \mbox{\rm and} \ \ \
W_n(X) := V_n(a_n X + b_n).
$$
\end{definition}

\begin{remark} It is clear from the recurrence property of the Tchebychev polynomials
given above that $T_n(X)$ is even (resp., odd) if $n$ is even (resp., odd). Thus $T_n^2(X)$
is always an even polynomial, so that the polynomial $V_n$ is well defined.
\end{remark}

\section{Three lemmas}

\begin{lemma}\label{01}
Let $f$ be a function that is continuous on $[0,1]$. Then
$$
\int_0^1 f(W_n(x)) \ \mbox{\rm d}x = A_n(f) \cos\frac{2\pi}{n} - B_n(f) \sin\frac{2\pi}{n}\cdot
$$
\end{lemma}

\begin{proof}
We make the change of variables $a_n x + b_n = \cos^2 t$, so that $t \in [\frac{\pi}{2n}, \frac{\pi}{n}]$
and $a_n$ d$x = -2 \sin t \cos t$ d$t = -\sin 2t$ d$t$. Thus
$$
\begin{array}{lll}
\displaystyle a_n\int_0^1 f(W_n(x)) \ \mbox{\rm d}x 
&= \displaystyle a_n\int_0^1 f(V_n(a_n x + b_n)) \ \mbox{\rm d}x 
&= \displaystyle -\int_{\frac{\pi}{2n}}^{\frac{\pi}{n}} f(V_n(\cos^2 t)) \sin 2t \ \mbox{\rm d}t  \\[2em]
&= \displaystyle -\int_{\frac{\pi}{2n}}^{\frac{\pi}{n}} f(T_n^2(\cos t)) \sin 2t \ \mbox{\rm d}t
&= \displaystyle -\int_{\frac{\pi}{2n}}^{\frac{\pi}{n}} f(\cos^2 nt) \sin 2t \ \mbox{\rm d}t 
\end{array}
$$
Putting $t = \frac{\pi}{n} - u $ in the last integral yields
$$
a_n\int_0^1 f(W_n(x)) \ \mbox{\rm d}x 
= - \int_0^{\frac{\pi}{2n}} f(\cos^2 nu) \sin \left(\frac{2\pi}{n} - 2u\right) \ \mbox{\rm d}u 
$$
which gives the result by expanding $\sin(\frac{2\pi}{n} - 2u)$.  
\end{proof}

\begin{lemma}\label{u0}
Let $f$ be a function that is continuous on $[0,1]$. Then
$$
\int_{\frac{1-b_n}{a_n}}^0 f(W_n(x)) \ \mbox{\rm d}x = - A_n(f).
$$
\end{lemma}

\begin{proof}
We make the same change of variables as in Lemma~\ref{01} above, obtaining
$$
a_n \int_{\frac{1-b_n}{a_n}}^0 f(W_n(x)) \ \mbox{\rm d}x 
= - \int_0^{\frac{\pi}{2n}} f(V_n(\cos^2 t)) \sin 2 t \ \mbox{\rm d}t
= - \int_0^{\frac{\pi}{2n}} f(\cos^2 nt) \sin 2 t \ \mbox{\rm d}t.
$$
\end{proof}

\begin{lemma}\label{uv}
Let $f$ be a function that is continuous on $[0,1]$. Then
$$
\int_{\frac{1-b_n}{a_n}}^{-\frac{b_n}{a_n}} f(W_n(x)) \ \mbox{\rm d}x = 
\begin{cases}
- A_n(f) - B_n(f) \frac{\cos \frac{\pi}{2n}}{\sin \frac{\pi}{2n}}, \  & \mbox{\rm if $n$ is odd;} \\[1em]
- 2 A_n(f) - 2 B_n(f) \frac{\cos \frac{\pi}{n}}{\sin \frac{\pi}{n}}, \ & \mbox{\rm if $n$ is even.}
\end{cases}
$$
\end{lemma}

\begin{proof}
Making once more the change of variable used in Lemma~\ref{01} above, we obtain
$$
a_n \int_{\frac{1-b_n}{a_n}}^{-\frac{b_n}{a_n}} f(W_n(x)) \ \mbox{\rm d}x = 
- \int_0^{\frac{\pi}{2}} f(\cos^2 nt) \sin 2t \ \mbox{\rm d}t
= - \sum_{k=0}^{n-1} I_{k,n}
$$
where
$$
I_{k,n} = \int_{\frac{k\pi}{2n}}^{\frac{(k+1)\pi}{2n}} f(\cos^2 nt) \sin 2t \ \mbox{\rm d}t.
$$
Now we will give another expression for $I_{k,n}$ according to the parity of $k$.

\begin{itemize}

\item If $k$ is odd, we make in $I_{k,n}$ the change of variable $t = \frac{(k+1)\pi}{2n} - u$.
This yields
$$
I_{k,n} = \int_0^{\frac{\pi}{2n}} 
            f(\cos^2(\tfrac{(k+1)\pi}{2} - nu)) \sin (\tfrac{(k+1)\pi}{n}-2u) \ \mbox{\rm d}u.
$$
But $k$ is odd, hence $\frac{(k+1)}{2}$ is an integer. Thus, expanding the sine, we obtain
$$
I_{k,n} = -a_n A_n \cos \left(\frac{(k+1)\pi}{n}\right) + a_n B_n \sin \left(\frac{(k+1)\pi}{n}\right).
$$
\item If $k$ is even, we make in $I_{k,n}$ the change of variable $t = \frac{k\pi}{2n}+u$.
This yields
$$
I_{k,n} = \int_0^{\frac{\pi}{2n}} 
            f(\cos^2(\tfrac{k\pi}{2} + nu)) \sin (\tfrac{k\pi}{n} + 2u) \ \mbox{\rm d}u.
$$
But $k$ is even, hence $\frac{k}{2}$ is an integer. Thus, expanding the sine, we obtain
$$
I_{k,n} = a_n A_n \cos \left(\frac{k\pi}{n}\right) + a_n B_n \sin \left(\frac{k\pi}{n}\right).
$$

\end{itemize}

We thus have 
$$
a_n \int_{\frac{1-b_n}{a_n}}^{-\frac{b_n}{a_n}} f(W_n(x)) \ \mbox{\rm d}x 
= - \sum_{k=0}^{n-1} I_{k,n} = \Sigma_1(n) + \Sigma_2(n)
$$
where
$$
\Sigma_1(n) = 
a_n A_n \sum_{\substack{0 \leq k \leq n-1 \\ k \ {\rm odd}}} \cos \left(\frac{(k+1)\pi}{n}\right) 
- a_n B_n \sum_{\substack{0 \leq k \leq n-1 \\ k \ {\rm odd}}} \sin \left(\frac{(k+1)\pi}{n}\right)
$$
and
$$
\Sigma_2(n) = 
- a_n A_n \sum_{\substack{0 \leq k \leq n-1 \\ k \ {\rm even}}} \cos \left(\frac{k\pi}{n}\right) 
- a_n B_n \sum_{\substack{0 \leq k \leq n-1 \\ k \ {\rm even}}} \sin \left(\frac{k\pi}{n}\right).
$$
Rearranging $\Sigma_1(n) + \Sigma_2(n)$ and simplifying by $a_n$ finally gives
$$
\int_{\frac{1-b_n}{a_n}}^{-\frac{b_n}{a_n}} f(W_n(x)) \ \mbox{\rm d}x =
\begin{cases}
- A_n 
- 2 B_n \displaystyle\sum_{\substack{1 \leq k \leq n-1 \\ k \ {\rm even}}} \sin \left(\frac{k\pi}{n}\right),
&\mbox{\rm if $n$ is odd;} \\
- 2 A_n 
- 2 B_n \displaystyle\sum_{\substack{1 \leq k \leq n-1 \\ k \ {\rm even}}} \sin \left(\frac{k\pi}{n}\right),
&\mbox{\rm if $n$ is even.}
\end{cases}
$$
To finish the proof of the lemma it suffices to make the classical computation
$$
\sum_{\substack{1 \leq k \leq n-1 \\ k \ {\rm even}}} \sin \left(\frac{k\pi}{n}\right)
= \Im \left(\sum_{\substack{1 \leq k \leq n-1 \\ k \ {\rm even}}} e^{\frac{ik\pi}{n}}\right)
=
\begin{cases}
\displaystyle\frac{\cos \frac{\pi}{2n}}{2 \sin \frac{\pi}{2n}} &\mbox{\rm if $n$ is odd,} \\[1.5em]
\displaystyle\frac{\cos \frac{\pi}{n}}{\sin \frac{\pi}{n}} &\mbox{\rm if $n$ is even.} \ \ \ 
\end{cases}
$$
\end{proof}

\section{The main result}

We are now ready to prove our main result. (Recall the definitions of $a_n$, $b_n$, $A_n(f)$, $B_n(f)$
and $W_n(X)$ given in Section~\ref{def} above.)

\begin{theorem}\label{main}
Let $n$ be an integer $> 2$, and $f$ a function that is continuous on $[0,1]$. Then
\begin{itemize}

\item if $n$ is odd, then
$$
\int_0^1 f(W_n(x)) \ \mbox{\rm d} x =
- \cos\frac{2\pi}{n} \int_{\frac{1-b_n}{a_n}}^{\frac{-b_n}{a_n}} f(W_n(x)) \ \mbox{\rm d}x
+\left(2\cos\frac{\pi}{n}-1\right) \int_0^{\frac{-b_n}{a_n}} f(W_n(x)) \ \mbox{\rm d} x
$$

\item if $n$ is even, then
$$
\int_{\frac{1-b_n}{a_n}}^1 f(W_n(x)) \ \mbox{\rm d} x =
\sin^2\frac{\pi}{n} \int_{\frac{1-b_n}{a_n}}^{\frac{-b_n}{a_n}} f(W_n(x)) \ \mbox{\rm d} x
$$

\end{itemize}

\end{theorem}

\begin{proof}
First we eliminate $B_n(f)$ between Lemma~\ref{01} and Lemma~\ref{uv}. Then we use Lemma~\ref{u0} to get
rid of $A_n(f)$. Finally, for $n$ even, we multiply by $\sin\frac{\pi}{n}$, divide by $2 \cos\frac{\pi}{n}$, and 
we combine two integrals, obtaining the statement above; for $n$ odd, we multiply by $\sin\frac{\pi}{2n}$ and we 
divide by $\cos\frac{\pi}{2n}$ obtaining the equality
$$
\int_0^1 f(W_n(x)) \ \mbox{\rm d} x =
4 \cos\frac{\pi}{n} \sin^2 \frac{\pi}{2n} \int_{\frac{1-b_n}{a_n}}^{\frac{-b_n}{a_n}} f(W_n(x)) \ \mbox{\rm d}x
+\left(1-2\cos\frac{\pi}{n}\right) \int_{\frac{1-b_n}{a_n}}^0 f(W_n(x)) \ \mbox{\rm d} x.
$$
The statement of the theorem follows by writing $\int_{\frac{1-b_n}{a_n}}^0 = \int_{\frac{1-b_n}{a_n}}^{\frac{-b_n}{a_n}}
- \int_0^{\frac{-b_n}{a_n}}$ and rearranging.  
\end{proof}

\begin{remark}

\ { }

\begin{itemize}

\item Theorem~\ref{main} above is still true, but trivial, for $n=1$ and $n=2$.

\item For $n = 3$, Theorem~\ref{main} gives
$$
\int_0^1 f(3x^2-2x^3) \ \mbox{\rm d} x = 
\frac{1}{2} \int_{-\frac{1}{2}}^{\frac{3}{2}} f(3x^2-2x^3) \ \mbox{\rm d} x
$$
which is exactly Ruehr's identity.

\item If $n=4$, then $a_4 = - \frac{\sqrt{2}}{4}$, and $b_4 = \frac{2+\sqrt{2}}{4}$. Thus, Theorem~\ref{main} gives
$$
\int_{1-\sqrt{2}}^1 f((x^2-2x)^2) \ \mbox{\rm d} x =
\frac{1}{2} \int_{1-\sqrt{2}}^{1+\sqrt{2}} f((x^2-2x)^2) \ \mbox{\rm d} x.
$$


\end{itemize}

\end{remark}

\end{document}